\documentclass[11pt]{article}

\usepackage[english]{babel}

\usepackage[a4paper,textheight=11cm]{geometry}

\usepackage{amsmath}
\usepackage{graphicx}
\usepackage[colorlinks=true, allcolors=blue]{hyperref}

\usepackage{amsmath,amscd,amsthm,makeidx,graphicx,url,bm,bbm,mathrsfs,palatino,bm,mathtools,latexsym}
\usepackage[bbgreekl]{mathbbol}

\usepackage{subfig}

\usepackage{amssymb}
\usepackage{a4wide,xy,graphicx}
\usepackage{tikz-cd}
\usepackage{enumerate,comment}
\usepackage{titlesec}
\usepackage{tikz-cd}

\numberwithin{equation}{section}

\makeindex

\xyoption{all}

\newtheorem{theorem}{Theorem}[section]
\newtheorem{proposition}[theorem]{Proposition}
\newtheorem{corollary}[theorem]{Corollary}

\newtheorem{lemma}[theorem]{Lemma}

\theoremstyle{remark}
\newtheorem{remark}[theorem]{Remark}

\theoremstyle{definition}

\makeatletter
\newcommand*\sbullet{\mathpalette\bigcdot@{.7}}
\newcommand*\bigcdot@[2]{\mathbin{\vcenter{\hbox{\scalebox{#2}{$\m@th#1\bullet$}}}}}
\makeatother

\def\beq{\begin{eqnarray}}
	\def\eeq{\end{eqnarray}}
\def\bes{\begin{eqnarray*}}
	\def\ees{\end{eqnarray*}}

\def\bP{{\mathbb P}}

\def\calE{{\mathcal E}}

\def\h{\mathfrak{h}}

\DeclareMathOperator{\Spec}{Spec} 
\DeclareMathOperator{\Hom}{Hom}

\def\bE{{\mathbb{E}}}
\def\bM{{\mathbb{M}}}

\def\A{{\mathbb A}}

\def\C{\mathbb{C}}

\def\calV{{\mathcal{V}}}

\def\calH{\mathcal{H}}
\def\calW{\mathcal{W}}

\def\n{{\mathbf{n}}}

\def\k{\mathfrak{k}}
\def\N{\mathbb{N}}

\def\t{\mathfrak{h}}
\def\m{\mathfrak{m}}
\def\calC{{\mathcal C}}
\def\calG{{\mathcal G}}
\def\calB{{\mathcal B}}
\def\calM{{\mathcal M}}
\def\calZ{{\mathcal Z}}
\def\m{{\mathfrak{m}}}

\def\calO{{\mathcal O}}

\def\sl{{\mathfrak s\mathfrak l}}

\def\m{{\mathfrak m}}
\newcommand{\p}{\mathfrak p}

\newcommand{\nc}{\newcommand}

\def\calZ{\mathcal{Z}}

\newcommand{\g}{\mathfrak{g}}
\renewcommand{\n}{\mathfrak{n}}
\renewcommand{\a}{\mathfrak{a}}

\usepackage{leftidx}

\newcommand{\calR}{\mathcal{R}}
\newcommand{\calS}{\mathcal{S}}

\newcommand{\Gr}{\textnormal{Gr}}

\newcommand{\End}{\textnormal{End}}

\newcommand{\opp}{\textnormal{op}}

\nc{\op}[1]{\mathop{\mathchoice{\mbox{\rm #1}}{\mbox{\rm #1}}
		{\mbox{\rm \scriptsize #1}}{\mbox{\rm \tiny #1}}}\nolimits}
\nc{\al}{\alpha}

\nc{\ep}{\varepsilon} 
\nc{\ga}{\gamma} 
\nc{\Ga}{\Gamma}
\nc{\la}{\lambda} 
\nc{\La}{\Lambda} 
\nc{\si}{\sigma}
\nc{\Sig}{{\Gamma}} 
\nc{\Om}{\Omega} 
\nc{\om}{\omega}

\nc{\SL}{\mathrm{SL}} 
\nc{\GL}{\mathrm{GL}} 
\nc{\SO}{\mathrm{SO}} 
\nc{\Sp}{\mathrm{Sp}} 
\nc{\PSp}{\mathrm{PSp}}

\nc{\PGL}{\mathrm{PGL}}
\nc{\G}{\mathrm{G}}

\nc{\W}{\mathrm{W}}
\nc{\Lg}{\mathrm{L}}
\nc{\Pg}{\mathrm{P}}
\nc{\calL}{{\mathcal L}}
\nc{\Sym}{{\rm Sym}}

\nc{\Frob}{\mathrm{Frob}}
\nc{\spec}{{\rm Spec}}

\nc{\rN}{\mathrm{N}}

\usepackage{longtable}

\nc{\cpt}{{\op{cpt}}} \nc{\Dol}{{\op{Dol}}} \nc{\DR}{{\op{DR}}}
\nc{\B}{{\op{B}}} \nc{\Triv}{\op{Triv}} \nc{\Hod}{{\op{Hod}}}
\nc{\Log}{{\op{Log}}} \nc{\Exp}{{\op{Exp}}} \nc{\Est}{E_{\op{st}}}
\nc{\Hst}{H_{\op{st}}} \nc{\Left}[1]{\hbox{$\left#1\vbox to
		10.5pt{}\right.\nulldelimiterspace=0pt \mathsurround=0pt$}}
\nc{\Right}[1]{\hbox{$\left.\vbox to
		10.5pt{}\right#1\nulldelimiterspace=0pt \mathsurround=0pt$}}
\nc{\LEFT}[1]{\hbox{$\left#1\vbox to
		15.5pt{}\right.\nulldelimiterspace=0pt \mathsurround=0pt$}}
\nc{\RIGHT}[1]{\hbox{$\left.\vbox to
		15.5pt{}\right#1\nulldelimiterspace=0pt \mathsurround=0pt$}}

\nc{\bee}{{\bf E}} 

\title{Center of Kostant algebra}
\author{Tam\'as Hausel}

\begin{document}
\maketitle
{\begin{center}{En l'honneur de G\'erard Laumon}\end{center}}
\begin{abstract}
In this note, following Mui\'c--Savin \cite{muic-savin},  we compute the center of Kostant algebra, introduced in \cite{kostant} as strongly commuting algebra. We explain how it encodes information on tensor products between a finite-dimensional and a Verma module, and  about the structure of principal series representations via the work of Bernstein--Gelfand \cite{bernstein-gelfand}. We also discuss conjectured relationships to Langlands duality.
\end{abstract}

Our main motivation is  to bridge the geometry of the Hitchin integrable system \cite{hitchin-stable} arising in mathematical physics to the Langlands program in number theory. Ng\^{o}'s proof of the fundamental lemma \cite{ngo}, following Laumon and Ng\^{o}'s \cite{laumon-ngo} proof in the unitary case, gave a spectacular example of such connections, but there are hints already in \cite{Laumon}. 

In \cite{hausel-thaddeus} we observed that the Hitchin systems for $\SL_n$ and Langlands dual $\PGL_n$ satisfy the prescription of Strominger--Yau--Zaslow \cite{strominger-yau-zaslow} for mirror symmetry. As a test of this observation, we formulated the topological mirror symmetry conjecture and signalled the hope to relate to the geometric Langlands correspondence of Drinfeld and Laumon \cite{Laumon0}. 


The topological mirror symmetry conjecture was first settled in \cite{groechenig-etal1} using p-adic integration techniques and then by  \cite{maulik-shen1} using Ng\^o's geometric techniques from \cite{ngo} as suggested by \cite{hausel-global}. In turn, \cite{groechenig-etal2} used their p-adic integration technique to reprove Ng\^o's geometric stabilisation result from \cite{ngo}, reversing the logic proposed in \cite{hausel-global}.

As we will discuss it in some more detail in Section~\ref{mirror}, the motivation for this paper comes from \cite{hausel-hitchin}, where very stable Higgs bundles, generalising Drinfeld and Laumon's  very stable bundles in \cite{Laumon}, were studied using elementary Hecke transformations. It is explained in \cite{hausel-ICM} how these can be understood using minuscule multiplicity algebras, and in turn \cite{hausel-big} introduced big algebras to generalise these ideas to general Hecke transformations. Here we will show that the medium algebras of \cite{hausel-big}, which are important subalgebras of the big algebras, naturally lead to representation theory relevant to the local Langlands program at the complex place. In future work  we will explore this representation theory from the perspective of big algebras. 

One can consider our approach  somewhat orthogonal to Ng\^o's, in that we are studying the geometry of Lagrangian multi-sections of the Hitchin fibration, while Ng\^o and Laumon were studying the Hitchin fibers. 

\noindent{\bf Acknowledgement.} The author would like to thank G\'erard Laumon for his contagious enthusiasm for geometric approaches to the Langlands program. We benefited from comments of and discussions  with Anne-Marie Aubert, Mischa Elkner, Oscar Garc\'ia-Prada, Nigel Higson, Jakub L\"owit, Shon Ng\^{o} and Catharina Stroppel.   The author's research is supported by an FWF grant ``Geometry of the tip of the global nilpotent cone'' number P 35847 and by an ERC advanced grant ``Representation theory, equivariant topology and Langlands duality via fixed point schemes" number 101199663. Incidentally, the global nilpotent cone was introduced in \cite{Laumon} and the "Langlands duality" is meant to be the semi-classical limit of \cite{Laumon0}.

\section{ Kostant's strongly commuting algebras}
\label{kostants}
For a semisimple complex Lie algebra $\g$ with Cartan subalgebra $\t\subset \g$, universal enveloping algebra $U(\g)$ and finite-dimensional irreducible representation  $\pi^\mu:U(\g)\to \End(V^\mu)$ of highest weight $\mu\in \Lambda^+\subset \Lambda \subset \t^*$  dominant, integral weight, Kostant introduced  \cite[(4.6.2)]{kostant} the strongly commuting algebra \begin{align}\label{def}
\calR^\mu(\g):=(U(\g)\otimes\End(V^\mu))^\g\end{align} as an associative algebra over the center $Z(\g):=U(\g)^\g=Z(U(\g))\subset U(\g)$. If \begin{align}\label{delta} \delta:U(\g)\to U(\g)\otimes \End(V^\mu),\end{align} defined by the condition $\delta(x)=x\otimes 1+1\otimes \pi^\mu(x)$ for $x\in \g$, denotes the diagonal map  then the $\g$-invariant subring $\calR^\mu(\g)$ of $U(\g)\otimes \End(V^\mu)$ can be thought to be the commutant of $\delta(U(\g))$: $$\calR^\mu(\g)=C_{U(\g)\otimes \End(V^\mu)}(\delta(U(\g)))\subset U(\g)\otimes \End(V^\mu).$$  

 By \cite[Proposition 2.13]{higson} (see Lemma~\ref{longhigson} below), Kostant algebras coincide, in the complex Lie algebra case, with the Hecke algebras appearing in \cite{muic-savin} and \cite{higson}, going back at least to   \cite{lepowsky} and \cite{knapp-vogan}; but in a way as far as \cite{harish-chandra} and \cite{parthasarathy-rao-varadarajan} in our case of complex Lie algebra $\g$.  Therefore, their representation theory plays an important role in classifying admissible irreducible representations of a connected complex Lie group $\G$, with Lie algebra $\g$, considered as a real Lie group. This classification is part of the local Langlands program at the complex place.  We can also mention \cite{joseph}, where the Kostant algebras were considered by the name relative Yangians, and their representation theory was studied from the perspective of  the Bernstein-Gelfand-Gelfand category $\calO$ \cite{humphreys}. Kostant algebras were also studied in the recent \cite{sun-zhao}.

Rozhkovskaya \cite[Theorem 4.1]{rozhkovskaya} proved that $\calR^\mu(\g)$ is commutative if and only if $V^\mu$ is weight multiplicity free, i.e., if $\dim(V^\mu_\lambda)\leq 1$ for any weight $\lambda\in\Lambda$. Kostant \cite[Remark 4.9]{kostant} observed that $\delta(Z(\g))\subset Z(\calR^\mu(\g))$ is in the center of $\calR^\mu(\g)$. We will show below in Theorem~\ref{explicit} that they generate the center $Z(\calR^\mu(\g))$ over $Z(\g)$. 

We denote the $Z(\g)$-subalgebra $$\calZ^\mu(\g):=\langle \delta(x)_{x\in Z(\g)} \rangle_{Z(\g)}\subset \calR^\mu(\g),$$ given by generators, the {\em filtered medium algebra.} Thus $\calZ^\mu(\g)$ is generated by $Z(\g)$ and $ \delta(Z(\g))$. Consequently, we can compute $\Spec(\calZ^\mu(\g))$ by embedding it in \begin{align} \label{ambient} \Spec(Z(\g))\times \Spec(Z(\g)) \cong \t^*/\!/W\!\sbullet\times\, \t^*/\!/W \sbullet,\end{align} where the Weyl group $W$ acts on $\t^*$ via the $\sbullet$-action, i.e. $w\sbullet v = w(v+\rho)-\rho$, where $v\in \t^*$, $w\in W$ and $\rho\in \Lambda\subset  \t^*$ is the sum of fundamental weights, or equivalently the half-sum of positive roots. In \eqref{ambient} we used the Harish-Chandra isomorphism $Z(\g)\cong U(\t)^{W\sbullet}\cong S(\h)^{W\sbullet}$, where $S(\h)$ denotes the symmetric algebra of $\h$.

First, we  consider  the fiber $\calR^\mu_\chi(\g):=\calR^\mu(\g)/(\ker(\chi))$ for a character $\chi:Z(\g)\to \C \in \Spec(Z(\g))$, where $(\ker(\chi))\triangleleft \calR^\mu(\g)$ is the left ideal generated by $\ker(\chi)$. Let $U_\chi:=U(\g)/(\ker(\chi))$ denote the maximal quotient of $U(\g)$. Then we have the following
\begin{lemma}\begin{align} \label{have} \calR^\mu_\chi(\g)=\calR^\mu(\g)/(\ker(\chi))=(U_\chi\otimes \End(V^\mu))^\g.
\end{align} \end{lemma}
\begin{proof}
The second equation above follows, because the natural map $\calR^\mu(\g)\to (U_\chi\otimes \End(V^\mu))^\g
$ has kernel $(\ker(\chi))$, thus we have the embedding \begin{align}\label{embed} \calR^\mu(\g)/(\ker(\chi))\hookrightarrow(U_\chi\otimes \End(V^\mu))^\g.\end{align} Additionally, we determine their dimensions
\begin{align}\label{dim} \dim(\calR^\mu(\g)/(\ker(\chi)))=\sum_{\lambda\in \Lambda} (m^\mu_\lambda)^2=\dim(\End(V^\mu)^\h)=\dim((U_\chi\otimes \End(V^\mu))^\g).
\end{align}
The first equation follows because by \cite[Proposition 2.1]{rozhkovskaya} as $Z(\g)$-modules we have $$\calR^\mu(\g)\cong_{Z(\g)} \sum_{\lambda\in\Lambda} Mat_{m^\mu_\lambda}(Z(\g)),$$ where $m^\mu_\lambda:=\dim(V^\mu_\lambda)$ is the dimension of the weight spaces in the decomposition as an $\t$-module $$V^\mu\cong_\t \bigoplus_{\lambda\in \Lambda} V^\mu_\lambda.$$ The second equation of \eqref{dim} is straightforward. The third follows from \begin{align} \label{kostant} U(\g)\cong_{Z(\g)} Z(\g)\otimes H(\g)\end{align} being \cite[Theorem 0.13]{kostant-lie} a free $Z(\g)$-module. While as a $\g$-module under the adjoint  action \eqref{kostant} matches $U(\g)$ with the trivial $\g$-action on $Z(\g)$ and  $$H(\g)\cong_\g \bigoplus_{\nu\in \Lambda^+}  (V^\nu)^{m^\nu_0}.$$  
Now \eqref{have} follows from \eqref{embed} and \eqref{dim}. 
\end{proof}
Let $\lambda\in \t^*$ be such that the infinitesimal character  \begin{align}\label{character}\chi_\lambda:Z(\g)\to \C\end{align} of the Verma module $M_\lambda$, of highest weight $\lambda$, satisfies $\chi_\lambda=\chi$. By the Duflo-Joseph theorem \cite[Corollary 18.8]{etingof}, cf.  \cite[6.2.(8)]{jantzen}, we have $U_{\chi_\lambda}\cong \Hom_{fin}(M_\lambda,M_\lambda)$, where $\Hom_{fin}$ is the $\g$-finite part of $\Hom$ (cf. \cite[18.1]{etingof}), i.e. those homomorphisms which are contained in finite-dimensional $\g$-subrepresentations. Then we get \begin{align}\label{verma} \calR^\mu_\chi(\g) &=(U_\chi\otimes \End(V^\mu))^\g=(\Hom_{fin}(M_\lambda,M_\lambda)\otimes \End(V^\mu))^\g\nonumber \\ &=(\End_{fin}(M_\lambda\otimes V^\mu))^\g=(\End(M_\lambda\otimes V^\mu))^\g=\End_\calO(M_\lambda\otimes V^\mu)\end{align}
as the endomorphism algebra of $M_\lambda\otimes V^\mu$ in the BGG category $\calO$ \cite{humphreys}. The fourth equation follows because the $\g$-invariant part is automatically in the  $\g$-finite part. 


Let $S^\mu:=\{\mu_i\}\subset \Lambda$ denote the set of weights in $V^\mu$. Then the set  of infinitesimal characters appearing in $M_\lambda\otimes V^\mu$ is $\{\chi_{\lambda+\mu_i}\}$
by \cite[Theorem 3.6]{humphreys}. Because $M_\lambda\otimes V^\mu$ has a Verma filtration whose subqotients are $M_\lambda+\mu_i$ on which $\delta(Z(\g))$ acts by scalar via the character $\chi_{\lambda+\mu_i}$  Thus the spectrum of the image of $ \delta(Z(\g))$ in $\calR^\mu_\chi$ is  supported at $\{[\lambda+\mu_i]_{\sbullet}\}\subset \t^*/\!/W\sbullet=\Spec(Z(\g))$, where $$[\lambda+\mu_i]_{\sbullet}:=W{\sbullet}(\lambda+\mu_i)\in \t^*/\!/W{\sbullet}.$$ Using this and arguments of \cite{muic-savin} we will prove the following.

\begin{theorem}\label{explicit}
\begin{enumerate} \item The filtered medium algebra $\calZ^\mu(\g)=Z(\calR^\mu(\g))$ coincides with the center of the Kostant algebra.\item $\Spec(\calZ^\mu(\g))\subset\Spec(Z(\g))\times \Spec(Z(\g))$ is the reduced subscheme with  $$\Spec(\calZ^\mu(\g))(\C)= \{ ([\lambda]_{\sbullet},[\lambda+\mu_i]_{\sbullet}) \mid \lambda\in \t^*, {\mu_i}\in S^\mu \}\subset \Spec(Z(\g))\times \Spec(Z(\g)).$$ In other words $\calZ^\mu(\g)\cong Z(\g)\otimes Z(\g)/I_\mu$ where $Q\in I_\mu\triangleleft Z(\g)\otimes Z(\g)\subset \C[\t^* \oplus \t^*] $ if and only if $Q(\lambda,\lambda+\mu_i)=0$ for all $\lambda\in \t^*$ and $\mu_i\in S^\mu$. \end{enumerate}
\end{theorem}
\begin{remark} Note that the ideal $I_\mu$ and ring $Z(\g)\otimes Z(\g)/I_\mu$ appear in \cite[Theorem 2.5]{bernstein-gelfand} without the connection to the center of Kostant algebra $\calR^\mu(\g)$.
\end{remark}

\begin{proof} Here we recall the ingredients of the proof, in our case of a complex Lie algebra $\g$, from \cite{muic-savin}. They \cite{muic-savin} prove the analogous results more generally for quasi-split real groups. We will use results of \cite{lepowsky-mccollum} and \cite{lepowsky} as discussed in \cite{dixmier} and \cite{wallach}, which were developed in the general real semisimple Lie group case, and applied in the case of complex semisimple Lie groups in \cite{duflo} and \cite{higson}.

Following \cite{duflo} we let $\g_\C:= \g\otimes\C = \g_1\oplus\g_2$ be the complexification of the Lie algebra $\g$, where $\g_1\cong\g_2\cong \g$. 
Let \begin{align}\label{iwasawa}\g_\C=\k\oplus\a\oplus\n\end{align}
 be the Iwasawa decomposition. Here $\k\cong\g$ is the twisted diagonal copy of $$\k=(x,-{}^t x)_{x\in \g} \subset \g_1\oplus \g_2,$$ $$\a=(x, x)_{x\in\t}\subset \g_1\oplus\g_2,$$ is a maximal commutative subalgebra in the subspace $$\p=\a+\n=(x,{}^t x)_{x\in \g}<\g_1\oplus \g_2.$$ We also denote by $$\m=(x,-x)_{x\in \t}$$ the Cartan subalgebra of the twisted diagonal $\k$. We denoted $x\mapsto {}^t x$ a Chevalley anti-involution of $\g$, which is defined using a Chevalley basis $\{e_\beta,f_\beta,h_i\in \h: \beta\in \Delta^+, 1\leq i\leq r\}$ by ${}^t e_\beta=f_\beta$ and ${}^ t h_i=h_i$ with respect to a choice of $\h\subset \g$ Cartan subalgebra and set of positive roots. 
 
 In particular, $\a\cong \t\cong \m$ canonically. We can thus write $\t^*_\C\cong \t^*\oplus \t^*\cong \a^*\oplus \m^*$ in two different ways. We will either use coordinates $$(\xi,\psi)\in \t^*\oplus \t^*$$ or $$\nu\oplus \lambda \in \a^*\oplus \m^*.$$ Thus,   \begin{align}\label{twoway}(\xi,\psi)=\nu\oplus \lambda\end{align} means that $\xi+\psi=\nu$ and $\xi-\psi=\lambda$, and $\xi=(\nu+\lambda)/2$ and $\psi=(\nu-\lambda)/2$.

Then the decomposition \eqref{iwasawa} induces the {\em canonical mapping} \cite[9.2.2]{dixmier}  \begin{align*}
    U(\g_\C)\to U(\a)\otimes U(\k),\end{align*} which when restricted to the commutant $U(\g_1\oplus \g_2)^\k$ of  the twisted diagonal $U(\k)\subset U(\g_1\oplus \g_2)$ induces \cite[9.2.3.iii]{dixmier} an anti-homomorphism: \begin{align}\label{canonical} p: U(\g_1\oplus\g_2)^\k \to U(\a)\otimes U(\k)^{\m}.
\end{align} 
When followed by  our finite-dimensional irreducible representation $\pi^\mu:U(\k)\to \End(V^\mu)$ then \cite[Theorem 1.3]{lepowsky} identifies the kernel of \begin{align*} p_\mu:=\pi^\mu\circ p: U(\g_1\oplus\g_2)^\k \to U(\a)\otimes \End(V^\mu)^{\m}
\end{align*} as \begin{align} \label{kernel}\ker(p_\mu)=U(\g_\C)^\k \cap U(\g_\C) I_\mu,\end{align}  where $I_\mu=\ker(\pi^\mu)\triangleleft U(\k)\subset U(\g_\C)$ is the kernel of the representation $\pi^\mu:U(\k)\to \End(V^\mu)$. Thus we can describe the Hecke algebra of \cite[\S 3]{muic-savin} in our special complex Lie algebra case, following \cite{higson}, in the following \begin{lemma} \label{longhigson}\begin{align}\label{list} U(\g_\C)^\k/\ker(p_\mu) \nonumber &=U(\g_\C)^\k/(U(\g_\C)^\k \cap U(\g_\C) I_\mu)\\ 
\nonumber & =(U(\g_\C)\otimes_{U(\k)}\End(V^\mu))^\k\\ \nonumber &=(U(\g_1)\otimes (\End(V^\mu))^{op})^\g\\ &=\calR^{\mu^{\!*}}(\g_1).\end{align}
\end{lemma}
\begin{proof}
In the first equation we used \eqref{kernel}. For the second first note the product \begin{align}\label{product} (a_1\otimes b_1) \cdot (a_2\otimes b_2)=a_1 a_2\otimes b_2 b_1\end{align} for $a_1\otimes b_1,a_2\otimes b_2\in U(\g_\C)\otimes_{U(\k) }\End(V^\mu)$ is well defined on $(U(\g_\C)\otimes_{U(\k)}\End(V^\mu))^\k$ the $U(\k)$-commutant part. Then because $\pi^\mu:U(\k)\twoheadrightarrow\End(V^\mu)$ is surjective so is the $U(\g_\C)$-module map \begin{align}\label{pitilde}\tilde{\pi}:U(\g_\C)\twoheadrightarrow U(\g_\C)\otimes_{U(\k)} \End(V^\mu)\end{align} and in turn we have the surjective ring homomorphism\begin{align*}\pi:U(\g_\C)^\k\twoheadrightarrow (U(\g_\C)\otimes_{U(\k)} \End(V^\mu))^\k.\end{align*} To compute the kernel of $\pi$ we note that the kernel of $\tilde{\pi}$ is $U(\g_\C)I_\mu$ by the following Lemma applied to the case $R=U(\g_\C)$, $A=U(\k)$ and $I=I_\mu$ (cf. also \cite[3.5.2.(1)]{wallach}). \begin{lemma}
     Let $I\triangleleft A\subset R$ be a left ideal in a subring of a unital ring. Then as left $R$-modules we have \begin{align*} R/RI\cong R\otimes_A A/I. 
    \end{align*}
\end{lemma} \begin{proof} First we can define $f:R/RI\to R\otimes_A A/I$ by $f(r+RI):= r\otimes 1$ and $g:R\otimes_A A/I\to R/RI$ by $g(r\otimes (a+I)):=ra+RI$. Then it is straightforward to check that $f$ and $g$ are well-defined and are inverses to each other. The result follows.
\end{proof} Now it follows that $\ker(\pi)=U(\g_\C)^\k\cap U(\g_\C)I_\mu$ which incidentally agrees with $U(\g_\C)^\k\cap U(\g_\C)I_\mu=U(\g_\C)^\k\cap I_\mu U(\g_\C)$ by \cite[9.1.10.(ii)]{dixmier}, thus is a two-sided ideal of $U(\g_\C)^\k$. This altogether implies the second equation of \eqref{list}. 

 In the third equation of \eqref{list},  we denoted $$\g=(x,x)\subset \g_1\oplus \g_2$$
the diagonal,  and we used the crossed product $U(\g_\C)=U(\g_1)\rtimes U(\k)$ following \cite[Proposition 2.13]{higson}. The opposite algebra appears because of the natural product structure on $(U(\g_\C)\otimes_{U(\k)}\End(V^\mu))^\k$  in \eqref{product}. In fact, the surjection $U(\g_\C)^\k\twoheadrightarrow \calR^{\mu^{\!*}}\!(\g)$ obtained this way is \cite[Lemma 2.15]{higson}, which is different from the coordinatewise one $$1\otimes \pi^{\mu^{\!*}}:U(\g_\C)^\g=(U(\g_1)\otimes U(\g_2))^\g\twoheadrightarrow (U(\g_1)\otimes \End(V^{\mu^{\!*}}))^\g.$$ Finally, in the fourth equation in \eqref{list} we used that the opposite matrix algebra $\End(V^\mu)^{op}\cong \End((V^\mu)^*)\cong \End(V^{\mu^*}),$ where $\mu^*$ the highest weight of the dual representation $(V^\mu)^*$. This completes the proof of Lemma~\ref{longhigson}.
\end{proof}

Thus, by \eqref{kernel} 
and \eqref{list} we get an injective anti-homomorphism:
\begin{align}\label{heckeanti} p_\mu: \calR^{\mu^{\!*}}\!(\g)\hookrightarrow U(\a)\otimes \End(V^\mu)^\m
\end{align}

We have $\End(V^\mu)^\m\cong \oplus_\lambda\End(V^\mu_\lambda)$, where $V^\mu_\lambda$ denotes the $\h$-weight space  $\lambda\in \Lambda(\g)$ of the representation $V^\mu$. We denote $\dim(V^\mu_\lambda)=m^\mu_\lambda$. As $U(\a)\cong S^*(\a)$ canonically, for each $\nu\in \a^*\cong \t^*$ and $\lambda\in \Lambda\subset \t^*$
such that $m^{\mu}_\lambda\neq 0$ we have a representation of $\calR^{\mu^{\!*}}\!(\g)$ on $(V^\mu_\lambda)^*$ denoted \begin{align}\label{rho}\rho_{\nu,\lambda}:=p_\mu\otimes \C_\nu:\calR^{\mu^{\!*}}\!(\g)\to \End((V^\mu_\lambda)^*).\end{align} 

Note that if $X$ is a Harish-Chandra $(\g_\C,\k)$-module, then the finite-dimensional $\Hom_\g(V^\mu,X)$ is an $\calR^{\mu^{\!*}}(\g_1)$-module. Moreover, the map \begin{align}\label{correspondence}X\mapsto \Hom_\g(V^\mu,X)\end{align}establishes (see \cite[Theorem 5.5]{lepowsky-mccollum} or a more recent discussion \cite[Proposition 2.10]{higson}) a bijection between the set of irreducible Harish-Chandra $(\g_\C,\k)$-modules $X$ with $$\Hom_\g(V^\mu,X)\neq 0$$ and irreducible representations of $\calR^{\mu^{\!*}}\!(\g)$.  
\begin{proposition} \label{prop} For $\nu\in \a^*\cong \h^*$ and $\lambda\in \Lambda\subset \m^*\cong\h^*$ let $(\xi,\psi)=\nu\oplus\lambda$ as in \eqref{twoway}. Then  the representations $\rho_{\xi,\psi}:=\rho_{\nu,\lambda}$ in \eqref{rho} satisfy the following properties. When $\xi$ is $\sbullet$-dominant, that is $\langle \xi+\rho, \alpha\rangle$ is not a negative integer for any $\alpha\in \Delta^+$ positive root, and $\psi$ is $\sbullet$-antidominant, that is $$-\sbullet \psi:=-\psi-2\rho$$ is $\sbullet$-dominant, we have \begin{enumerate} \item  \label{first}$\rho_{\xi,\psi}$ is irreducible.
\item \label{second} For $w\in W$ we have $\rho_{w\sbullet \xi,w\sbullet\psi}\cong\rho_{\xi,\psi}.$

\item \label{third}  The irreducible  $\rho_{\xi,\psi}$ is a subquotient of $\rho_{\xi,w\sbullet\psi}$ for any $w\in W$ such that $w\sbullet\psi-\psi \in \Lambda_r\subset \Lambda$ is in the root lattice.
\end{enumerate}
\end{proposition}
\begin{proof} Note that by \cite[3.5.8]{wallach} the principal series Harish-Chandra $(\g_\C,\k)$-module (cf. \cite[\S 6]{jantzen})  \begin{align}\label{principal}X(\xi,\psi):=\Hom_{fin}(M_{\psi},M^\vee_\xi)\end{align} corresponds to $\rho_{\xi,\psi}$ under the map \eqref{correspondence}. Here $M_{\xi}$ is the Verma module of highest weight $\xi\in \t^*$, $M^\vee_{\psi}$ is the dual Verma module (cf. \cite[\S 3.2]{humphreys}) and $\Hom_{fin}$ is the $\k$-finite part of $\Hom$ (cf. \cite[18.1]{etingof}). 

Now \ref{first}.  follows from  \cite[Theorem 4.4]{duflo}. \ref{second}. follows from \ref{first}. and  \cite[Lemme 5.5]{duflo}. Finally, \ref{third}. (and also \ref{first}.) can be seen  from the equivalence of categories in \cite[\S 6]{bernstein-gelfand}, (cf. \cite[\S 6]{jantzen}), which for $\psi$ $\sbullet$-antidominant translates the statements from $X(\xi,\psi)$ to the corresponding statements about $M_{\xi}^\vee$. Ultimately, it follows from \cite[Theorem 6.7]{bernstein-gelfand}.
\end{proof}

Now we observe that \begin{align}\label{centerconst} p_\mu(Z(\calR^{\mu^{\!*}}\!(\g)))\subset Z(U(\a)\otimes \End(V^\mu)^\t)\cong U(\a)\otimes \bigoplus_{\lambda\in S^\mu} Z(\End((V^\mu_\lambda)^*)).\end{align} It is because $\rho_{\nu,\lambda}$ is irreducible for a Zariski dense set of $\nu$'s  by Proposition~\ref{prop}. Thus, the center $Z(\calR^{\mu^{\!*}}\!(\g))_\nu\subset Z(\calR^{\mu^{\!*}}_\nu\!(\g))$ will act with a scalar for a Zariski dense subset of $\nu$'s in $\a^*$. As $Z(\End((V^\mu_\lambda)^*))\subset \End((V^\mu_\lambda)^*)$ is a closed subset we get \eqref{centerconst}.

Now the subset $$\calV^\mu:= \{(\psi+\lambda,\psi)| \psi\in \t^*, \lambda\in S^\mu\}\subset \t^*\times \t^*$$ is a union of affine hyperplanes and thus is an affine subvariety. It also parametrizes the principal series representations
$X(\psi+\lambda,\psi)$  such that $\Hom_\g(V^\mu,X(\psi+\lambda,\psi))\neq 0$. In turn, we can identify $\calV^\mu\cong \Spec( U(\a)\otimes \bigoplus_{\lambda\in S^\mu} Z(\End((V^\mu_\lambda)^*)))$ in the obvious way, and see that the map $p_\mu(z)$ is just the function on $\calV^\mu$ given by assigning to the principal series $X(\psi+\lambda,\psi)$ the scalar, with which the central element $z\in Z(\calR^{\mu^{\!*}})$ acts on it. By Proposition~\ref{prop}.\ref{second}, this function is invariant under the diagonal $W\sbullet$-action. By Proposition~\ref{prop}.\ref{third}, $p_\mu(z)$ is invariant under the $W\sbullet$-action on the second coordinate of  $\t^*\times \t^*$ in the sense that it takes the same value on $(\xi,\psi)$ and $(\xi,w\sbullet \psi)$ as long as both $(\xi,\psi)\in \calV^\mu$ and $(\xi,w\sbullet \psi)\in \calV^\mu$. This means that we can extend our function on $\calV^\mu$ uniquely to a $W\sbullet\times W\sbullet$-invariant function  on $$\widetilde{\calV}^\mu:=(1\times W\sbullet)(\calV^\mu)=\{(\psi+\lambda,w\sbullet \psi)| \psi\in \t^*, \lambda\in S^\mu, w\in W\}\subset \t^*\times \t^*$$
Thus we see that $$p_\mu:Z(\calR^{\mu^{\!*}})\hookrightarrow \C[\widetilde{\calV}^\mu]^{W\sbullet \times W\sbullet}.$$
To show surjection we note that $$\widetilde{\calV}^\mu/\!/ (W\sbullet \times\, W\sbullet )\subset \t^*/\!/ W\sbullet\times\, \t^*/\!/ W\sbullet \cong \Spec(Z(\g)\otimes Z(\g))\cong \Spec(Z(\g_\C))$$ is an affine subvariety, because $\widetilde{\calV}^\mu$
 is a closed $W\sbullet \times\, W\sbullet$-stable  finite union of affine subspaces.  Thus the restriction map \begin{align}\label{surj} Z(\g_\C)\twoheadrightarrow \C[\widetilde{\calV}^\mu]^{W\sbullet \times W\sbullet} \end{align} is surjective. We can conclude the proof of Theorem~\ref{explicit} by noting the restriction map \eqref{surj} factors through the map $Z(\g_\C)\subset Z(U(\g_\C)^\k)\to Z(\calR^{\mu^{\!*}}\!(\g))$. \end{proof}
\begin{corollary}
Let $f$ denote the map on centers of the surjective map $U(\g_\C)^\k \twoheadrightarrow \calR^{\mu^*}(\g)$ in \eqref{list} as $f:=Z(\g_\C)\hookrightarrow Z(U(\g_\C)^\k) \to Z(\calR^{\mu^*}(\g))$. Writing \begin{align*} f:Z(\g_\C)=Z(\g_1\oplus\g_2) = Z(\g_1)\otimes Z(\g_2)\to \calZ^{\mu^{\!*}}\!(\g).\end{align*} we have that  $z_1\otimes z_2\in Z(\g_1)\otimes Z(\g_2)$ maps to \begin{align*}
    f(z_1\otimes z_2)=z_1\delta(z_2)\in \calR^{\mu^{\!*}}\!(\g_1).\end{align*}  In particular, $f:Z(\g_\C) \twoheadrightarrow \calZ^{\mu^*}(\g)$ is surjective. 
    \end{corollary} 
    \begin{proof}
Let
$$
q:U(\g_\C)^\k \twoheadrightarrow
\left(U(\g_\C)\otimes_{U(\k)}\End(V^\mu)\right)^\k
\cong \calR^{\mu^{\!*}}(\g_1)
$$
be the quotient map appearing in \eqref{list}. We compute the restriction of
$q$ to the two central factors
$$
Z(\g_\C)=Z(\g_1\oplus\g_2)=Z(\g_1)\otimes Z(\g_2).
$$
Identifying $\g_1,\g_2$ with $\g$, and identifying
$\k$ with $\g$ by
$x\longmapsto (x,-{}^t x),$
we have, for $y\in \g_2\simeq \g$,
$$
(0,y)=({}^t y,0)+(-{}^t y,y),
$$
where $(-{}^t y,y)\in \k$. Hence we have
$$
q(0,y) = {}^t y\otimes 1+1\otimes \pi^\mu(-{}^t y)
\in U(\g_1)\otimes \End(V^\mu)^{\opp},
$$
under the  identification
$$
\End(V^\mu)^{\opp}\cong \End((V^\mu)^*)=\End(V^{\mu^{\!*}}),
$$ where
an endomorphism \(A\in \End(V^\mu)\) acts on \((V^\mu)^*\) by
\(\varphi\mapsto \varphi\circ A\). Therefore
$\pi^\mu(-{}^t y)$ gets identified with $\pi^{\mu^{\!*}}({}^t y),$
because the dual infinitesimal action is given by
$$\pi^{\mu^{\!*}}(x)\varphi=-\varphi\circ \pi^\mu(x).$$
Consequently,
$ q(0,y) = {}^t y\otimes 1+1\otimes \pi^{\mu^{\!*}}({}^t y)=\delta({}^t y)
\in \calR^{\mu^{\!*}}(\g_1).$

It follows that, for $z_2\in Z(\g_2)$,
$$ q(z_2)=\delta({}^t z_2)=\delta(z_2), $$
because the Chevalley anti-involution fixes the center $Z(\g)$.
For $z_1\in Z(\g_1)$, the quotient map gives 
$ q(z_1)=z_1\otimes 1. $
Thus we can conclude
$$
f(z_1\otimes z_2) = q(z_1)q(z_2) = z_1\delta(z_2) \in \calR^{\mu^{\!*}}(\g_1).$$
Finally by Theorem~\ref{explicit} $\calZ^{\mu^{\!*}}(\g)$ is generated over $Z(\g)$ by
$\delta(Z(\g))$  thus the map
$ f:Z(\g_\C)\twoheadrightarrow \calZ^{\mu^{\!*}}(\g) $
is surjective.
\end{proof}
\begin{remark}\label{harishchandra}  The proof of Theorem~\ref{explicit} above is the straightforward application of the usual strategy \cite[Theorem 7.4.5]{dixmier} originally employed for the Harish-Chandra isomorphism: $Z(\g)\cong S(\t)^{W\bullet}$. Namely, there we proceed by  considering the projection $U(\g)\to U(\h)$ induced by the triangular decomposition $\g=\overline{{\mathfrak n}}+\t+{\mathfrak n}$. Analogously, in the proof above, we had the canonical mapping \eqref{canonical} induced from the Iwasawa decomposition \eqref{iwasawa}. Next, one observes that the restriction $Z(\g)\hookrightarrow U(\h)\cong S(\h)$ is an embedding, and identifies the image using the linkage principle for Verma modules. Similarly we had the embedding \eqref{centerconst}, and we identified the image by understanding the linkage for principal series representations in Proposition~\ref{prop}.\end{remark}
\begin{remark} Note that if $\mu,\lambda\in \Lambda^+$ such that $\mu-\lambda$ is a non-negative linear combination of positive roots then $S^\lambda\subset S^\mu$ thus $\widetilde{\calV}^\lambda \subset \widetilde{\calV}^\mu$ and so we have canonical surjections: $\calZ^\mu(\g)\twoheadrightarrow \calZ^\lambda(\g)$. We denote $$\tilde{\calV}:=\{(\psi,\xi)| \psi-w\xi\in \Lambda \mbox{ for some } w\in W \}\subset \t^*\times \t^*$$ the ind-scheme $\widetilde{\calV}=\varinjlim \widetilde{\calV}^\mu$. In light of Theorem~\ref{explicit} we have the following corollary, which  is \cite[Theorem 2]{muic-savin} specialised to the complex Lie algebra case. 
\end{remark}
\begin{corollary} \label{center} The center $Z(\g_\C,\k)$ of the category of Harish-Chandra $(\g_\C,\k)$-modules is 
$$Z(\g_\C,\k)\cong \varprojlim \calZ^\mu(\g) \cong \C[\widetilde{\calV}]^{W\sbullet\times W\sbullet}.$$
\end{corollary}

\begin{remark} For a $W$-orbit   $[\mu_i]:=W\mu_i\in S^\mu/W$ we have an irreducible component  $\Spec(\calZ^\mu(\g))_{[\mu_i]}$ of $\Spec(\calZ^\mu(\g))$ given by \begin{align}\label{component} 
 \Spec(\calZ^\mu(\g))_{[\mu_i]}(\C)= \{ ([\lambda]_{\sbullet},[\lambda+\mu_i]_{\sbullet}) \mid \lambda\in \t^*\} \end{align} which can be identified with $\t/\!/ W_{\mu_i}\sbullet$, which in turn can be identified with $\Spec(H^*_{\G^\vee}(\G^\vee/P^\vee_{\mu_i}))$. In \cite{hausel-loewit} this will be related to the fixed point set of the loop rotation on $\Gr^\mu_{\G^\vee}$ being $\coprod_{[\mu_i]\in S^\mu/W} \G^\vee/P^\vee_{\mu_i}$ motivated by Nakajima's observation \cite[Theorem 3.1]{hausel-big} relating $\calZ^\mu(\g)$ to the equivariant cohomology of $\Gr^\mu_{\G^\vee}$.
\end{remark}

\begin{corollary}\label{symmetry} The Cartan  involution $\theta:\g_\C\to \g_\C$ sending $(x,y)$ to $(-{}^t y,-{}^t x)$ induces \begin{align}\label{involution} \Phi:([\lambda]_{\sbullet},[\mu]_{\sbullet})\mapsto ([-\sbullet\mu]_{\sbullet},[-\sbullet\lambda]_{\sbullet})=([-\mu-2\rho]_{\sbullet},[-\lambda-2\rho]_{\sbullet})\end{align} on $\Spec(Z(\g))\times \Spec(Z(\g))$, which in turn defines an involution on $\calZ^\mu(\g)$ swapping the coefficients $Z(\g)$ with the corresponding generators $\delta(Z(\g))$. 
\end{corollary}

\begin{remark} When $-1\in W$ (i.e. in types $A_1$,$B$,$C$,$D_{2n}$, $G_2$,$F_4$, $E_7$ and $E_8$) then $-\sbullet\in W\sbullet$ and so the symmetry \eqref{involution} is just $([\lambda]_{\sbullet},[\mu]_{\sbullet})\mapsto ([\mu]_{\sbullet}, [\lambda]_{\sbullet})$.
\end{remark}

\begin{remark} We note that $\Phi$ preserves any  component \eqref{component}. Where it induces the involution $$\begin{array}{cccc} \Phi_{\mu_i}:&\C[\t^*]^{W_{\mu_i}\sbullet}&\rightarrow&\C[\t^*]^{W_{\mu_i}\sbullet} \\ &f(\lambda)&\mapsto &f(-\sbullet(\lambda+\mu_i))\end{array}, $$ which swaps the coefficients $\C[\t^*]^{W\sbullet}\cong Z(\g)$ with the generators $\C[\t^*]^{W+\mu_i\sbullet}=\delta(Z(\g))\subset \C[\t^*]^{W_{\mu_i}\sbullet}$, where $W\!+\mu_i\sbullet$ denotes the $w\in W$ action on $\h^*$ given by $\lambda\mapsto w(\lambda+\mu_i+\rho)-\mu_i-\rho$ with unique fixed point at $-\rho-\mu_i$.
\end{remark}

\begin{remark} For  a self-dual $\mu^{\!*}=\mu$  we have $S^\mu=-S^\mu$ and then $-\sbullet$  acts on $\Spec(\calZ^\mu(\g))$ by $$-\sbullet:([\lambda]_{\sbullet},[\mu]_{\sbullet})\mapsto ([-\sbullet\lambda]_{\sbullet},[-\sbullet\mu]_{\sbullet}).$$
It is induced by the principal anti-involution \cite[2.2.18]{dixmier} on $U(\g_\C)$ extending $x\mapsto -x$ for $x\in \g_\C$. Thus $$-\sbullet\Phi: ([\lambda]_{\sbullet},[\mu]_{\sbullet})\mapsto ([\lambda]_{\sbullet},[\mu]_{\sbullet})$$ also acts on $\Spec(\calZ^\mu(\g))$. Together  $\{1,\Phi,-\sbullet,-\sbullet\Phi\}$  form a Klein group action on $\calZ^\mu(\g)$.
\label{klein}
\end{remark}

\section{Associated graded}
\label{graded}
Independently of \cite{kostant} Kirillov \cite{kirillov0}  reintroduced the algebras $\calR^\mu(\g)$ under the name {\em quantum family algebra}. He also introduced their classical counterpart $\calC^\mu(\g):=(S^*(\g)\otimes \End(V^\mu))^\g$ an associated, graded $S^*(\g)^\g$-algebra, which he called {\em classical family algebra}. For shortness, we will call $\calR^\mu(\g)$ {\em Kostant algebra} and $\calC^\mu(\g)$ {\em Kirillov algebra}.

The associated graded of the standard filtration $U(\g)=\cup_{p=0}^\infty U_p(\g)$ satisfies $\overline{U(\g)}\cong S^*(\g)$.  The standard filtration is $\g$-invariant, thus it induces a filtration  on $\calR^\mu(\g)$ which satisfies 
\begin{align}\label{assgr}\overline{\calR^\mu(\g)}\cong {\calC}^\mu(\g).\end{align}  We recall the symmetrization map \begin{align}\label{sym}\omega:S^*(\g)\stackrel{\cong}{\to} U(\g)\end{align} and the standard vector space grading $U(\g)=\oplus_{p=0}^\infty U^p(\g)$ induced by $\omega$ and the natural grading on $S^*(\g)$. The following lemma describes how the diagonal map $\delta$ in \eqref{delta} works on the associated graded. 

\begin{lemma}\label{lemma} If $f\in S^p(\g)$ then write $$\delta(\omega(f))=(\delta^q(\omega(f)))_{q=0}^p \in \bigoplus_{q=0}^p (U^q(\g)\otimes \End(V^\mu)).$$ 
Then we have \begin{align}\label{leading}\delta^p(\omega(f))=\omega(f)\otimes 1\in U(\g)\otimes \End(V^\mu), \end{align} \begin{align}\label{linear}\overline{\delta^{p-1}(\omega(f))}=\pi^\mu(d(f))\in S^*(\g)\otimes \End(V^\mu),\end{align} where $d(f)\in S^*(\g)\otimes (\g^*)^*\cong S^*(\g)\otimes \g$ denotes the exterior derivative of the polynomial function $f:\g^*\to \C$. 
\end{lemma}

\begin{proof} Let $\{x_1,\dots,x_d\}$ be a basis of $\g$ and $\nu\in \N^d$ such that $|\nu|:=\nu_1+\dots+\nu_d=p$. Let 
${x}^\nu:=\frac{{x}_1^{\nu_1}\dots {x}_d^{\nu_d}}{\nu_1!\dots \nu_d!}\in U(\g)$ while its associated graded $\overline{x}^\nu:=\frac{\overline{x}_1^{\nu_1}\dots \overline{x}_d^{\nu_d}}{\nu_1!\dots \nu_d!}\in S^p(\g)$, where we used the notation $\overline{x}_i:=x_i\in \g$ to signify that it lives in $S^*(\g)$. 
Then we can determine \begin{align*}\delta(x^\nu)=\frac{(x_{1}\otimes 1+1\otimes \pi^\mu(x_{1}))^{\nu_1}}{\nu_1!}\cdots \frac{(x_{d}\otimes 1+1\otimes \pi^\mu(x_{d}))^{\nu_d}}{\nu_d!}.\end{align*} We see that \begin{align}\label{leadingproof}\delta^p(x^\nu)=x^\nu\otimes 1.\end{align} and that \begin{align*}\delta^{p-1}(x^\nu)=\sum_{\nu_i>0} \frac{{x}_1^{\nu_1}\dots {x_i^{\nu_i-1}}\dots {x}_d^{\nu_d}}{\nu_1!\dots(\nu_i-1)!\dots \nu_d!}  \otimes \pi^\mu(x_i).\end{align*} Thus \begin{align}\label{forma}\overline{\delta^{p-1}(x^\nu)}= \sum_{\nu_i>0}\frac{\overline{x}_1^{\nu_1}\dots {\overline{x}_i^{\nu_i-1}}\dots \overline{x}_d^{\nu_d}}{\nu_1!\dots(\nu_i-1)!\dots \nu_d!}  \otimes \pi^\mu(x_i)=\sum_{i=1}^d \frac{\partial}{\partial \overline{x}_i}(\overline{x}^\nu)  \otimes \pi^\mu(x_i)=\pi^\mu(d(\overline{x}^\nu)).\end{align} As $\{x^\nu +U_{p-1}(\g)\mid |\nu|=p\}$  form a basis by \cite[Theorem 2.1.11]{dixmier} for $U_p(\g)/U_{p-1}(\g)$ \eqref{leadingproof} implies \eqref{leading} and \eqref{forma} implies \eqref{linear}. The result follows.
\end{proof}

Denote by $$\calM^\mu(\g):=\langle \{\pi^\mu(dx)\}_{x\in S(\g)^\G}\rangle_{S(\g)^\G}\subset \calC^\mu$$ the {\em graded  medium algebra} or {\em medium algebra} for short. Then Lemma~\ref{lemma} implies the following

\begin{proposition} \label{associated} The associated graded of the standard filtration on $\calR^\mu(\g)$ induces $\overline{\calZ^\mu(\g)}\cong \calM^\mu(\g).$
\end{proposition}

\begin{remark} To an invariant polynomial $C\in Z(\g)\cong U(\t)/\!/W\sbullet$ we can associate the quantum $M$-operator $$M_C=\delta(C)-C-C(\mu)\in \calZ^\mu(\g).$$ The natural filtration on $\calZ^\mu(\g)$ induced from the filtration on $U(\g)$ will be induced from $fil(C)=\deg(C)$ and $fil(M_C)=\deg(C)-1$. We see that $\Phi(C)=\delta(C\circ -\sbullet)$ and $\Phi(\delta(C))=C\circ -\sbullet$. Thus $\Phi(M_C)=(C\circ -\sbullet) -\delta(C\circ -\sbullet)-C(\mu)$. When $C$ is homogeneous around $-\rho$ of degree $k=\deg(C)$ then $C\circ -\sbullet=(-1)^kC$ and $$\Phi(M_C)=(-1)^{k-1} M_C+((-1)^{k-1}-1)C(\mu)$$ and $$\Phi(C)=(-1)^kC+(-1)^k M_{C}+C(\mu).$$ Thus we see that $\Phi$ on the associated graded $\overline{\calZ^\mu(\g)}=\calM^\mu(\g)$ induces $(-1)^{deg}$. Similarly $-\sbullet$ induces $(-1)^{deg+age}$ while $-\sbullet \Phi$ induces $(-1)^{age}$ on the graded medium algebra $\calM^\mu(\g)$, where $age(C)=0$ and $age(M_C)=1$. Thus, the Klein group symmetry of Remark~\ref{klein} on $\calZ^\mu(\g)$ quantizes the Klein group symmetry $\{1,(-1)^{deg},(-1)^{age},(-1)^{deg+age}\}$ on $\calM^\mu(\g)$ when $S^\mu=-S^{\mu}$.
\end{remark}

\begin{remark} \cite[Theorem 6.3]{higson} showed that the representations of the Kirillov algebra $\calC^\mu(\g)$ are related to the representations of the Cartan motion group $\G \ltimes \p$. Using its representation theory, which has been developed in \cite{champetier-delorme}, we expect that a similar analysis as in the Kostant algebra case in Section~\ref{kostants} will give an explicit presentation for the graded medium algebra $\calM^\mu(\g)$. We plan to return to this problem. 
\end{remark}

\section{Example of \texorpdfstring{$\g=\sl_2$}{g=sl2
}}

Let $\g=\sl_2$ and $\mu=5\varpi_1\in \Lambda$, where $\varpi_1\in \Lambda$ is the fundamental weight. Then let $C_2=\lambda(\lambda+2)$ be the Casimir operator on a representation of highest weight $\lambda\in \t^*$. Thus we can identify $Z(\g)\cong \C[C_2]$. We can define $$\widetilde{M}_1:=\delta(C_2)\in Z^{5\varpi_1}(\sl_2).$$  

To understand our Theorem~\ref{explicit} in this case, we first depict in Figure~\ref{lines} $\calV^{5\varpi_1}$ the lines $(\lambda,\lambda+\mu_i)$ in $\t^*\times \t^*$ for $\mu_i$ in the set of weights $S^{5\varpi_1}=\{-5,-3,-1,1,3,5\}$ of $V^{5\varpi_1}$. The set $\calV^{5\varpi_1}$ parametrizes principal series representations $X(\xi,\psi)$ which, as a $\k$-module, contain $V^{5\varpi_1}$. 

\begin{figure}[ht!] 
\begin{center}
  \includegraphics[width=7cm]{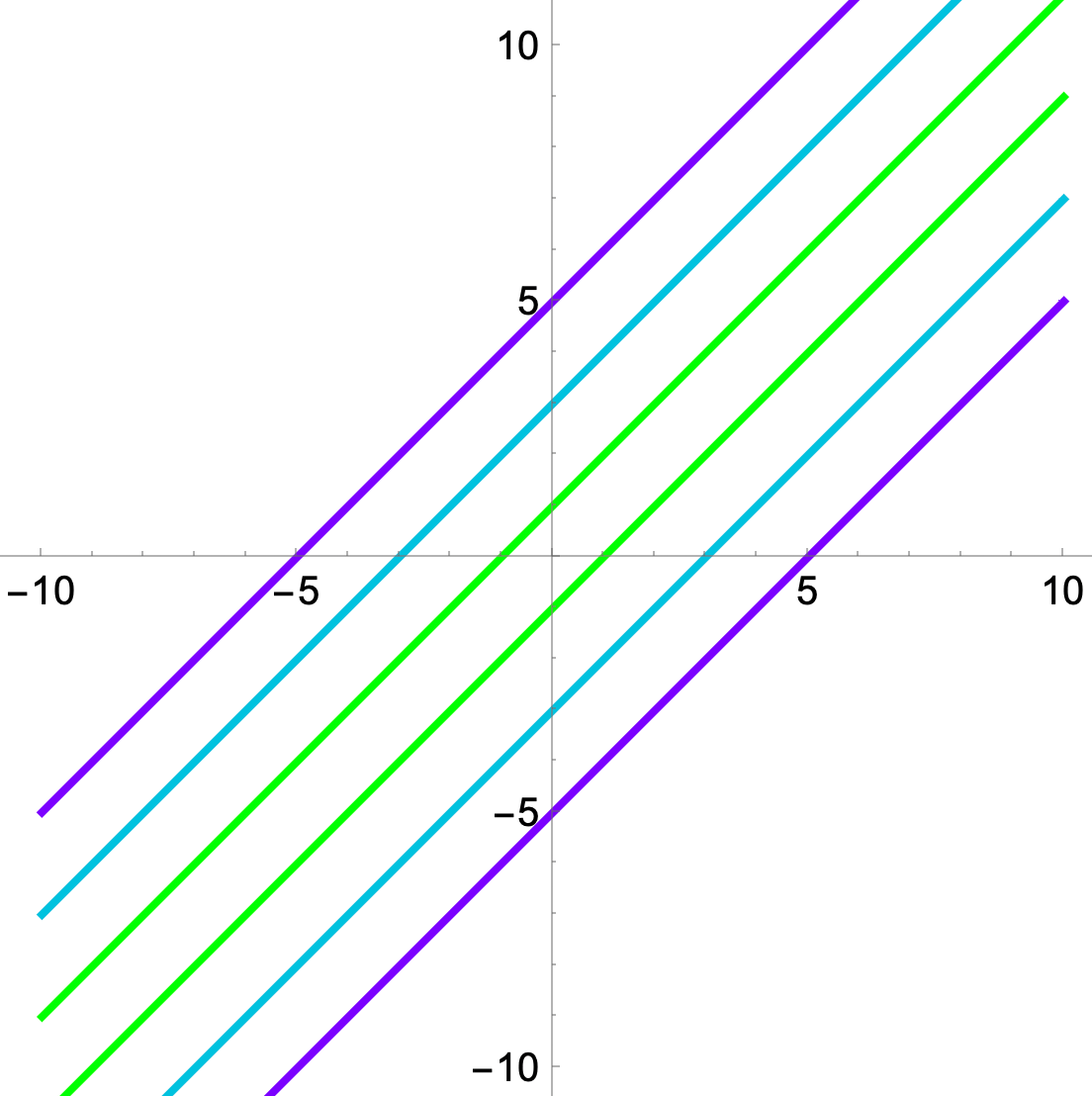}
\end{center}
\caption{\label{lines} The parameter space $\calV^{5\varpi_1}.$ }
\end{figure}

The image of these lines in the quotient $\t^*/\!/W\!\sbullet\times \t^*/\!/ W\sbullet$  in Figure~\ref{symmetricsl2} will be $\Spec(Z^{5\varpi_1})$ by Theorem~\ref{explicit}. Explicitely we get 
$Z^{5\varpi_1}=\C[C_2,\widetilde{M}_1]/((C_2^2 - 2C_2\widetilde{M}_1 + \widetilde{M}_1^2 - 2C_2 - 2\widetilde{M}_1 - 3)\\ (C_2^2 - 2C_2\widetilde{M}_1 + \widetilde{M}_1^2 - 50C_2 - 50\widetilde{M}_1 + 525)(C_2^2 - 2C_2\widetilde{M}_1 + \widetilde{M}_1^2 - 18C_2 - 18\widetilde{M}_1 + 45))$. Both this explicit form and the graph satisfy the symmetry $C_2 \leftrightarrow \widetilde{M}_1$ of Corollary~\ref{symmetry}.

\begin{figure}[ht!] 
\begin{center}
  \includegraphics[width=7cm]{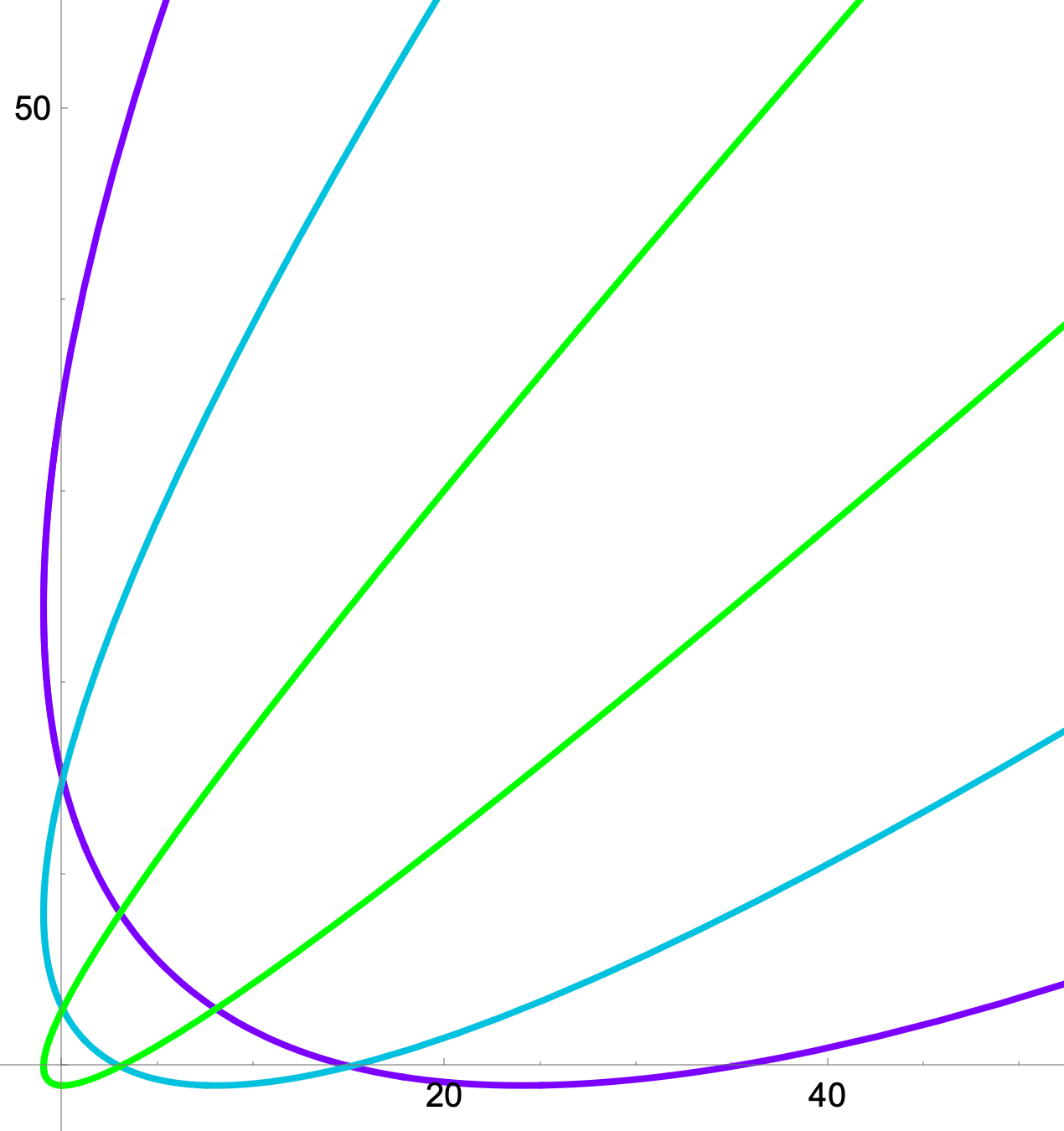}
\end{center}
\caption{\label{symmetricsl2}
 {$\Spec(Z^{5\varpi_1}(\sl_2))$.}}
\end{figure}

From   $\Spec(Z^{5\varpi_1} (\sl_2))$, using \eqref{verma}, we can read off the decomposition of $M_\lambda\otimes V^{5\varpi_1}$, for dominant $\lambda$ --- when $M_\lambda=P_\lambda$ by \cite[Proposition 3.8.(a)]{humphreys} --- into projective indecomposables as follows. One can compare these with the explicit formulas in \cite[\S 4]{troost}.

From Figure~\ref{symmetricsl2},    $\Spec(Z^{5\varpi_1} (\sl_2))\to \Spec(Z(\g))$ has singular fibers precisely at the discriminant locus $C_2=-1,0,3,8,15$, which correspond to highest weights $\lambda=-1,0,1,2,3$. Thus for other values of $C_2(\lambda)$ the tensor product $$M_\lambda\otimes V^\mu=\bigoplus_\nu M_\nu$$ will be decomposing as a direct sum of  Verma modules with distinct infinitesimal characters $C_2(\nu)$ given by the values of the graph of Figure~\ref{symmetricsl2} at such $\lambda$. 

While at $\lambda=-1$ we have $$M_{-1}\otimes V^{5\varpi_1}=P_{-6}\oplus P_{-4}\oplus P_{-2}.$$ The values of the graph at $-1$ are giving the infinitesimal characters $C_2(-2)=0,C_2(-4)=9,C_2(-6)=24$.

At $\lambda=0$ $$M_{0}\otimes V^{5\varpi_1}=P_{-5}\oplus P_{-3}\oplus M_{-1}\oplus M_{5},$$  their infinitesimal characters are $15,3,-1,35$ matching the values at $C_2=0$ of the graph in Figure~\ref{symmetricsl2}. 

At $\lambda=1$ we get $$M_1\otimes V^{5\varpi_1}=P_{-4}\oplus P_{-2}\oplus M_4\oplus M_6,$$ and the corresponding infinitesimal characters  $8,0,24,48$ matching the values of the graph. 

At $\lambda=2$ we get $$M_2\otimes V^{5\varpi_1}=P_{-3}\oplus M_{-1}\oplus M_3 \oplus M_5 \oplus M_7.$$
While at $\lambda=3$ we get $$M_3\otimes V^{5\varpi_1}=P_{-2}\oplus M_{2}\oplus M_4 \oplus M_6 \oplus M_8.$$ In both cases the infinitesimal characters appearing in the decomposition match the corresponding value of the graph. 

To compare with \cite[Proposition 5.1]{rozhkovskaya} we introduce 
$$M_1=\delta(C_2)-C_2\otimes 1 - 1\otimes C_2(5)\in Z^{5\varpi_1}(\sl_2)=\widetilde{M}_1-C_2\otimes 1 - 1\otimes C_2(5)\in Z^{5\varpi_1}(\sl_2).$$ 
Then rescaling $M_1$ and $C_2$ appropriately, Rozhkovskaya's formula reads that the minimal polynomial of $M_1$ over $\C[C_2]$ in $\calR^{k\varpi_1}(\sl_2)$ is 
$\prod_{j=0}^{k} (M-b_j)$ where $b_j=4(j^2 - n/2 - jn + (n/2 - j)\sqrt{C_2+1})$. In our case this gives:
$$Z^{5\varpi_1}(\sl_2)=\C[C_2,M_1]/(M_1^2 + 20M_1 - 100C_2)(M_1^2 + 52 M_1 - 36 C_2 + 640)(M_1^2 + 68 M_1 - 4 C_2 + 1152).$$ 

We depicted the graph in these coordinates in Figure~\ref{sl2}.

\begin{figure}[ht!] 
\begin{center}
  \includegraphics[width=7cm]{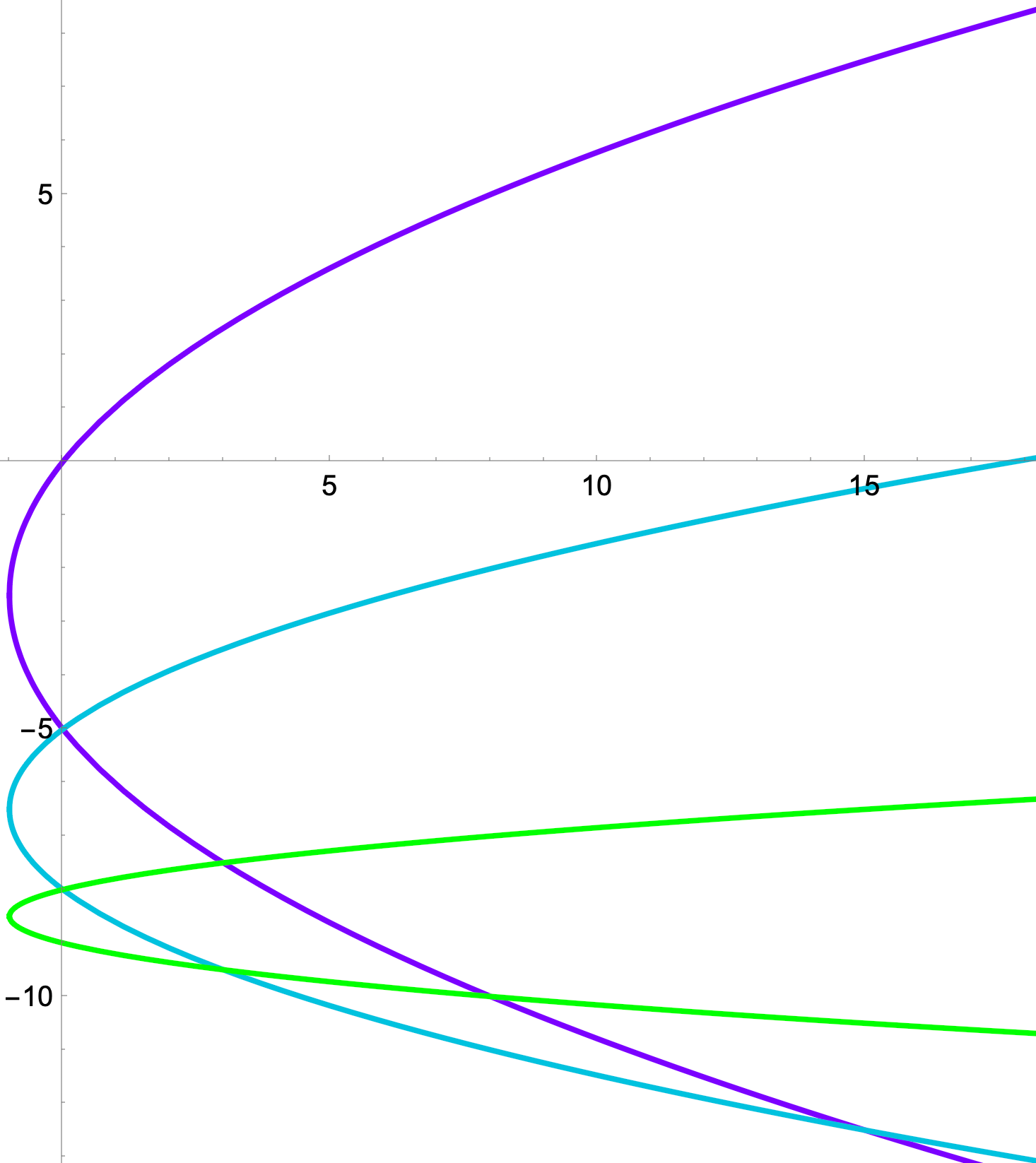}
\end{center}
\caption{\label{sl2}
 {$\Spec(Z^{5\varpi_1}(\sl_2))$.}}
\end{figure}

Finally, when we take in Proposition~\ref{associated} the associated graded $\calM^{5\varpi_1}(\sl_2)=\overline{Z^{5\varpi_1}(\sl_2))}$ of the standard filtration on $Z^{5\varpi_1}(\sl_2)$, then we get 

$$\calM^{5\varpi_1}(\sl_2)=\C[c_2,M_1]/(M_1^2  - 100c_2)(M_1^2  - 36 c_2 )(M_1^2  - 4 c_2).$$ 
The graph of which is depicted in Figure~\ref{kirillovsl2} (cf. also \cite[Fig 1]{hausel-big}). 

\begin{figure}[ht!] 
\begin{center}
  \includegraphics[width=7cm]{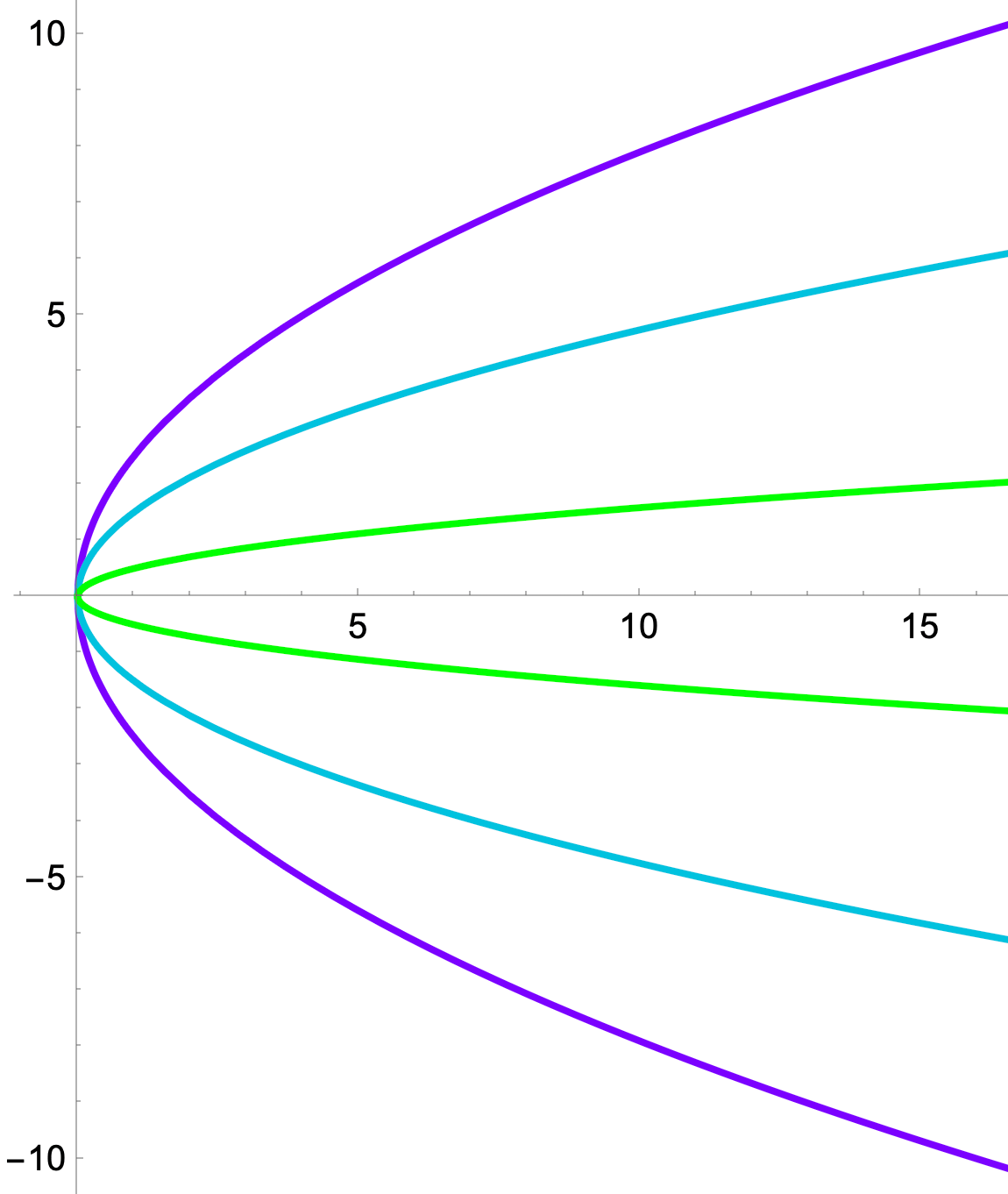}
\end{center}
\caption{\label{kirillovsl2}
 {$\Spec(\calM^{5\varpi_1}(\sl_2))$.}}
\end{figure}

As explained in \cite{hausel-ICM} one can see this picture as a model for the $\SL_2$-Hitchin system  on a certain Lagrangian multi-section (meaning  finite intersection with all fibers) of the Hitchin map. The fibers are all smooth, except at $0$ which corresponds to the nilpotent cone (cf. \cite{Laumon}).

Our last example in Figure~\ref{centerlorentz} is the odd component of $Z(\sl_2\times\sl_2,\sl_2)$ from Corollary~\ref{center} to see how the various centers $Z^{(2k+1)\varpi_1}(\sl_2)$ fit together. The intersection points of the parabolas correspond to  finite-dimensional irreducible representations of $\sl_2\times\sl_2$, while the other point at the resolution of the intersection corresponds to an infinite-dimensional simple  principal series representation which contains this finite-dimensional representation as a subquotient (cf. \cite[\S 12]{muic-savin}). 

Finally, the colouring from red to violet in the visible spectrum is showing the height of the "minimal type" of the principal series $X(\xi,\psi)$: the dominant weight $\mu\in \Lambda^+$ of lowest height such that $\Hom(V^\mu,X(\xi,\psi))\neq 0$. At the intersection points of the parabolas the two distinct simple subquotients have the minimal types indicated by the colours of the two branches.

\begin{figure}[ht!] 
\begin{center}
  \includegraphics[width=7.5cm]{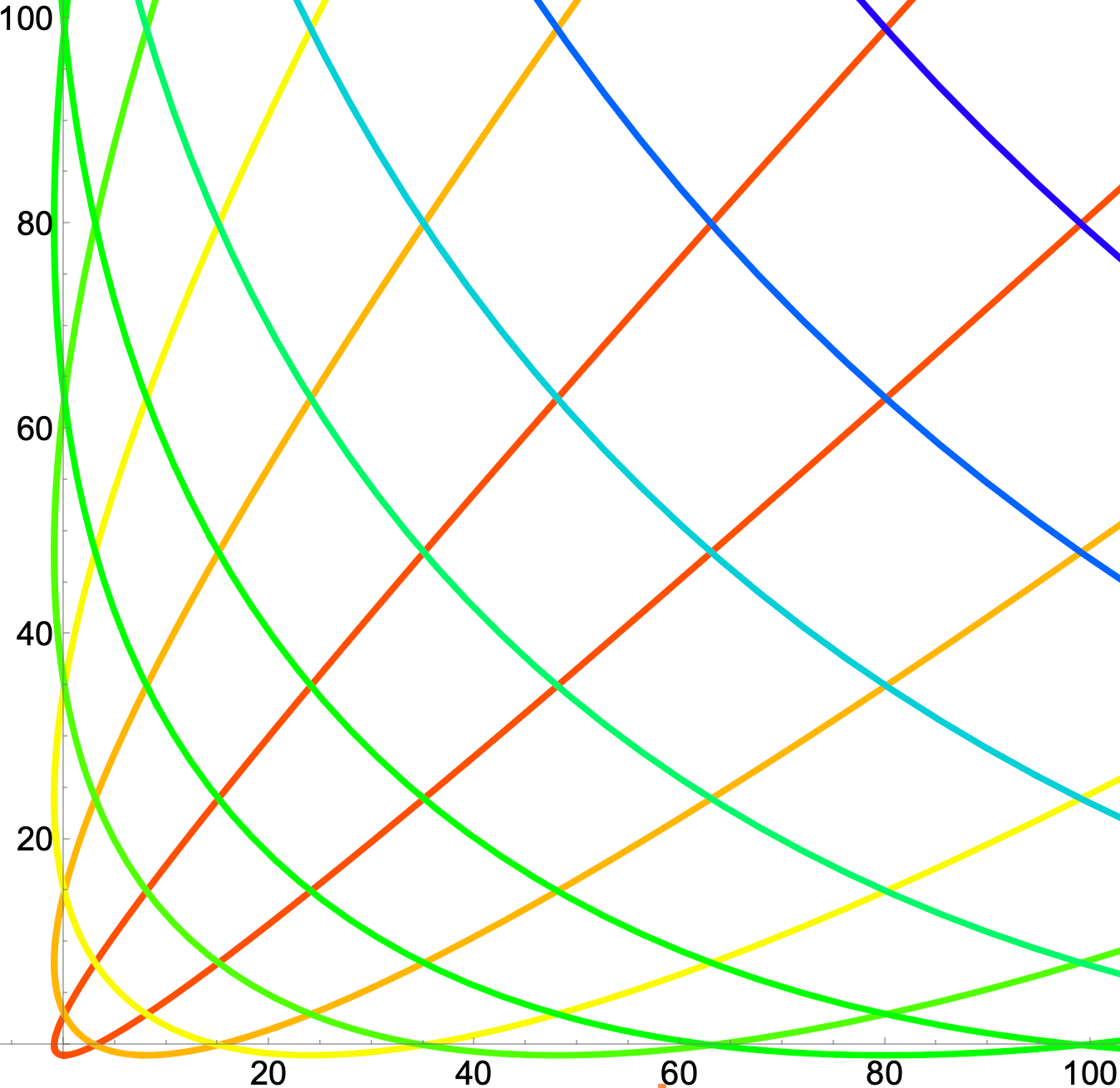}
\end{center}
\caption{\label{centerlorentz}
 {$\Spec(Z(\sl_2\times\sl_2,\sl_2)^-)$.}}
\end{figure}

\section{Conjectural relationships with Langlands duality}
\label{mirror}
The study of big and medium algebras \cite{hausel-big} was motivated by mirror symmetry considerations in \cite{hausel-hitchin,hausel-ICM,hausel-ICMtalk,hausel-simons}. In the semi-classical limit \cite{donagi-pantev} the geometric Langlands correspondence of Drinfeld--Laumon \cite{Laumon0} should give an equivalence of the form 
$$\calS:D^b({\mathbb M}_\G)\to D^b({\mathbb M}_{\G^\vee})$$ between the derived categories of coherent sheaves for Langlands dual Hitchin systems \begin{align*}\begin{array}{ccc}
    {\mathbb M}_\G && {\mathbb M}_{\G^\vee}
    \\ \!\! \! \!\! \!{\scriptstyle h_\G}{\downarrow} && \,\,\, \,\,\, \downarrow {\scriptstyle h_{\G^\vee}} \\ {\mathbb A}_\G &\cong &{\mathbb A}_{\G^\vee}  \end{array}.\end{align*} Here $\G$ is a complex reductive group with Langlands dual $\G^\vee$.  ${\mathbb M}_{\G}$ is the moduli space of $\G$-Higgs bundles  $(E,\Phi)$ on a smooth complex projective curve $C$, with $E$ principal $\G$-bundle on $C$ and $\Phi\in H^0(C;ad(E)\otimes K_C)$ is the canonical bundle $K_C$-twisted Higgs field. The Hitchin map \cite{hitchin-stable} is $h_\G:{\mathbb M}_{\G}\to {\A}_\G$ defined by evaluating invariant polynomials   $p\in \C[\g]^\G$, on $\g=Lie(\G)$ the Lie algebra of $\G$,  on the Higgs field $\Phi$. The generic fibers $h_{\G}^{-1}(a)$ and $h_{\G^\vee}^{-1}(a)$ should be \cite{hausel-thaddeus,donagi-pantev} dual Abelian varieties, and $\calS$ generically should be Fourier-Mukai transform. 

We expect \cite{donagi-pantev} that $\calS$ intertwines \begin{align}
\label{inter}\calH_c^\mu\circ \calS=\calS \circ \calW_c^\mu\end{align} the action of Wilson $\calW_c^\mu:D^b({\mathbb M}_\G)\to D^b({\mathbb M}_\G)$ and Hecke operators $\calH_c^\mu:=D^b({\mathbb M}_{\G^\vee})\to D^b({\mathbb M}_{\G^\vee})$. Here $c\in C$ a point, $\mu\in X^+(\G)$ is a dominant character, $\calW^\mu_c$ is defined by tensoring with the vector bundle $\rho^\mu(\bE_c)$, the universal $\G$-bundle $\bE$ on ${\mathbb M}_{\G}\times C$ restricted to ${\mathbb M}_{\G}\times \{c\}$ and evaluated in the $\mu$-highest weight representation $\rho^\mu:\G\to \GL(V^\mu)$. On the other hand $\calH_c^\mu$ is more technical to define, but heuristically, it should be given by Hecke transformations of Higgs bundles of type $\mu$ at the point $c$. When $\mu$ is minuscule in type $A$ this is defined \cite{hausel-hitchin} using the  elementary Hecke transformations of vector bundles on curves going back to \cite{narasimhan-ramanan}. 

As $\calS$ is generically a Fourier-Mukai transform, which maps the structure sheaf of an Abelian variety to the skyscraper sheaf at the identity of the dual Abelian variety, we expect that the mirror \begin{align}
\label{structure} \calS(\calO_{\bM_\G})=\calO_{W^+_0} \end{align} of the structure sheaf  $\calO_{\bM_\G}$ of $\bM_\G$ is the structure sheaf of $W^+_0\subset \bM_{\G^\vee}$ a section of the Hitchin map $h_{\G^\vee}$, called the {\em Hitchin section} \cite{hitchin-section}. 

Combining \eqref{structure} with \eqref{inter} we should get \begin{align} \label{test}\calH^\mu(\calO_{W^+_0})=\calS(\calW^\mu_c(\calO_{\bM_\G}))=\calS(\rho^\mu(\bE_c)).\end{align} In words: the  Hecke transformed Hitchin section should be the  mirror of the universal bundle in a representation restricted to a point.

In type A, and for minuscule $\mu$ \cite{hausel-hitchin} indicated that $\calH^\mu(\calO_{W_0^+})=\calO_{W^+_\mu}$ the structure sheaf of $W^+_\mu$, upward flow (the Bialinicky-Birula attracting set with respect to the natural $\C^\times$-action, scaling the Higgs field) from a certain very stable Higgs bundle $\calE_\mu$, and in turn checked the expectation \eqref{test} in this type $A$ minuscule case \begin{align} \label{case} \calS(\rho^\mu(\bE_c))=\calO_{W^+_\mu} \end{align}   generically using Fourier-Mukai transform. In turn, \cite{fang} checked \eqref{case}, in some sense over the whole Hitchin base.

Because $\calO_{W^+_\mu}$ has the structure of sheaf of algebras we can ask what  that structure induces on its mirror $\rho^\mu(\bE_c)$? From the property of the Fourier-Mukai transform mapping multiplication to convolution we should have a certain sheaf of convolution algebra structure on $\rho^\mu(\bE_c)$. But over the Hitchin section (generically the locus of the identity elements of the Abelian variety fibers) this structure should just be a sheaf of algebra structure on the restriction of $\rho^\mu(\bE_c)$ to the Hitchin section of $h_\G$. In turn, when $\mu$ is minuscule, this structure should be induced from the algebra structure on the  medium algebra $\calM^\mu(\g)\cong \calC^\mu(\g)$ or equivalently, because $\mu$ minuscule,  the  Kirillov algebra from Section~\ref{graded}.

Long story \cite{hausel-simons} short, the motivation to construct a sheaf of algebra structure on $\rho^\mu(\bE_c)$ for a general $\mu\in \Lambda^+$ along the Hitchin section led \cite{hausel-big} to the study of the Kirillov algebras $\calC^\mu(\g)$, their medium algebras $\calM^\mu(\g)$ and  certain maximal commutative big algebras $\calB^\mu(\g)$ in between \begin{align}\label{chain}\C[g]^\g \subset \calM^\mu(\g)\subset \calB^\mu(\g) \subset \calC^\mu(\g).\end{align} Morally, $\calB^\mu(\g)$ over $\C[g]^\g$ should model the mirror $\calS(\rho^\mu(\bE_c))$ over   the Hitchin base $\C[\A_{\G^\vee}]$ and $\Spec(\calM^\mu(\g))\to \Spec(\C[g]^\g)$ should model the support  $Supp(\calS(\rho^\mu(\bE_c)))\stackrel{h_{\G^\vee}}{\to} \A_{\G^\vee}$  over the Hitchin base, which is conjecturally \cite[Conjecture 3.19]{hausel-ICM} a multi-section of the Hitchin map, given by the Lagrangian closure of the upward flow $W^+_\mu$. For an example when $\g=\sl_2$ and $\mu=5\varpi_1$ see Figure~\ref{kirillovsl2}.

In \cite{hausel-big} a filtered analogue (in Loc. cit. it was called quantum analogue) of the chain of subalgebras \eqref{chain} was constructed $$Z(\g)\subset \calZ^\mu(\g)\subset \calG^\mu(\g) \subset \calR^\mu(\g)$$ in the Kostant algebra $\calR^\mu(\g)$, where $\calZ^\mu(\g)$ is the filtered medium algebra  of Section~\ref{kostants}. Finally, the maximal commutative subalgebra $\calG^\mu(\g)$ is the filtered big algebra of \cite{hausel-big}, ultimately constructed from the Feigin-Frenkel center \cite{feigin-frenkel}. While these filtered versions have interesting properties refining those of the classical counterparts in \eqref{chain}, it is not quite clear what  their role is in the mirror symmetry / Langlands duality picture we have been discussing in this section. One possibility is that they could be modelling Lagrangians in the de Rham version of the moduli spaces $\bM_\G$ like opers, closer to the original geometric Langlands correspondence \cite{Laumon0,beilinson-drinfeld}. Another possibility is that they model certain Lagrangians in certain irregular Higgs moduli spaces on $C=\bP^1$, which have an exotic $\C^\times$-action, like the ones studied in \cite{feigin-frenkel-toledano}.

Nevertheless, what we found in this note is that the filtered medium algebras $\calZ^\mu(\g)$ have deep connections to the representation theory of complex Lie groups, which are relevant in the local Langlands correspondence at the complex place. In future work we plan to explore this representation theory from the point of view of the filtered big algebras $\calG^\mu(\g)$. 

\bibliographystyle{alpha}
\bibliography{sample}

\end{document}